\documentclass{article}
\usepackage[utf8]{inputenc}                                                     
\usepackage{amsmath, amssymb, amsthm}
\usepackage{pythonhighlight}
\usepackage{mathrsfs}
\usepackage{graphicx}

\newcommand{\ZZ}{\mathbb{Z}}
\newcommand{\QQ}{\mathbb{Q}}

\newcommand{\PP}{\mathbb{P}}
\newcommand{\As}{\mathsf{A}_{1,0}}
\newcommand{\AMZV}{\mathsf{A}_\mathsf{MZV}}
\newcommand{\ep}{\varepsilon}

\newcommand{\sh}{\mathsf{K}}
\newcommand{\cat}{\operatorname{\mathbf{cat}}}
\newcommand{\sha}{\,\includegraphics[height=2mm]{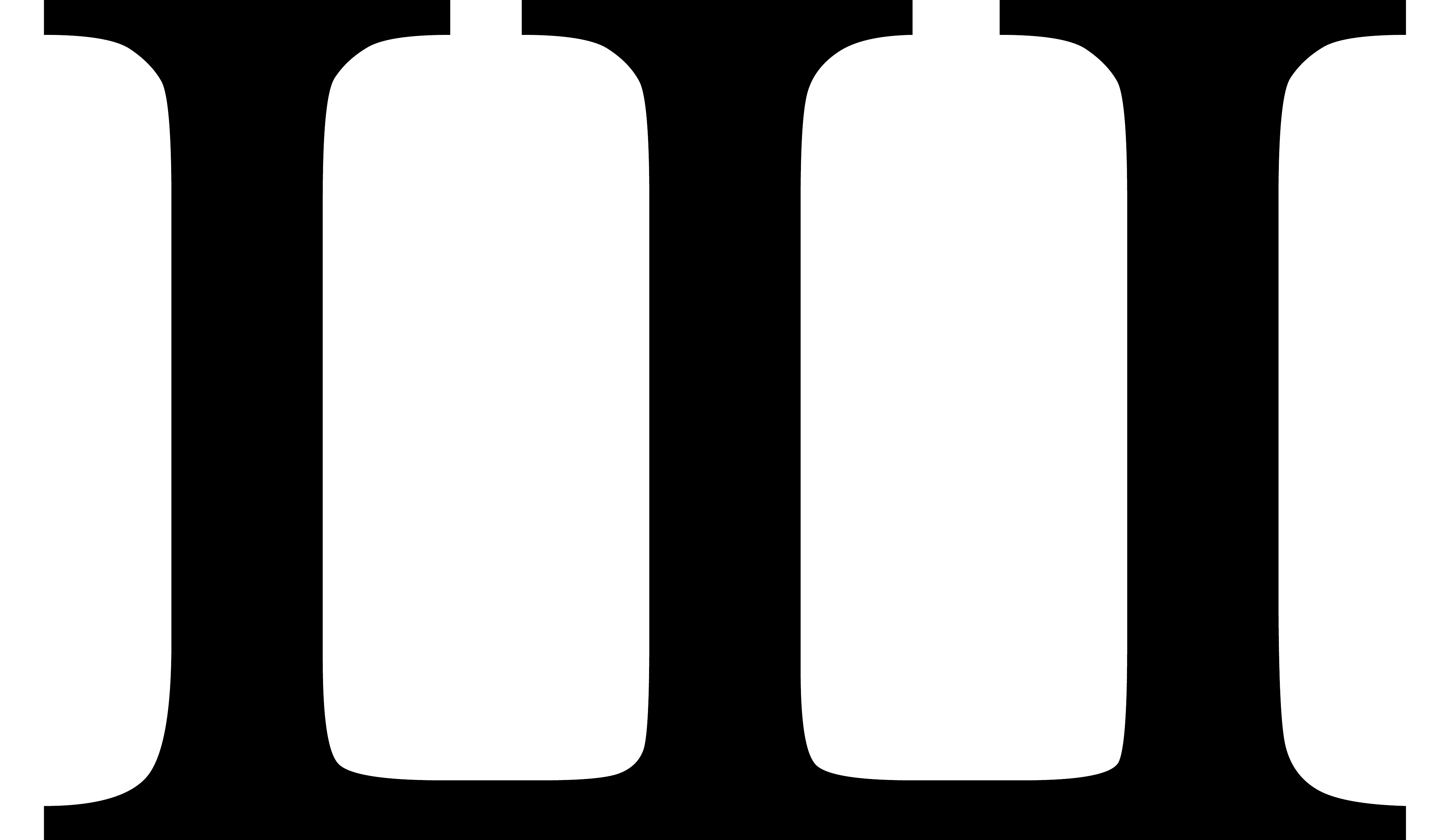}\,}

\newtheorem{thm}{Theorem}
\newtheorem{conjecture}{Conjecture}
\newtheorem{lemma}{Lemma}

\title{Zinbiel algebras and multiple zeta values}

\author{F. Chapoton}

\begin{document}

\maketitle

\section*{Introduction}

Multiple zeta values are the convergent iterated integrals from $0$ to
$1$ of the differential forms $\omega_0 = dt/t$ and
$\omega_1 = dt/(1-t)$. They form an algebra over $\QQ$, which has many
interesting connections with different domains, including knot theory
and perturbative quantum field theory \cite{waldschmidt, fresan}. This
algebra is expected to be graded by the weight, and a famous
conjecture of Zagier \cite{zagier} states that the dimensions of
homogeneous components are given by the Padovan numbers.

The algebra $\AMZV$ of motivic multiple zeta values is a more subtle
construction, in the setting of periods and mixed motives
\cite{brown_annals, brown_decomposition, fresan}. It can be defined as the
quotient of the commutative algebra $\As$, whose elements are seen as
formal iterated integrals of $\omega_0$ and $\omega_1$, by the
non-explicit ideal of all relations that can be proved using algebraic
geometry. This algebra is known to be graded by the weight and its
dimensions are given by the Padovan sequence, by results of
Brown \cite{brown_annals}.

There is a surjective morphism, called the period map, from the motivic algebra $\AMZV$ to the
usual algebra of multiple zeta values,
defined by taking the numerical value of a formal iterated
integral. This period map is expected to be injective, hence an
isomorphism.

\medskip

The aim of this article is to propose an algebraic construction, using
the algebraic structures known as zinbiel algebras or dual Leibniz
algebras, of some commutative sub-algebras $C_{u,v}$ of $\As$.

The notion of zinbiel algebras was introduced by Loday in relationship
with Leibniz algebras \cite{loday_dual_leibniz}. Their name is the reversal of
Leibniz, a play of words justified by the Koszul duality of the two
corresponding operads. Maybe a better name would be ``half-shuffle
algebras'', as they are very closely related to shuffle
algebras on words.

The algebras $C_{u,v}$, depending on algebraic parameters $u,v$, have
the same graded dimensions as the motivic algebra $\AMZV$.  Our main
conjecture is then that the restricted quotient map from $C_{u,v}$
to the motivic algebra $\AMZV$ is generically an isomorphism. The
special case when $u+v = 0$ seems to give instead an interesting
sub-algebra of $\AMZV$.

One interest of this construction is that it would provide new bases
of the algebra $\AMZV$ indexed by words in $2$ and $3$. Unlike the
known Hoffman basis indexed by the same set \cite{brown_annals, hoffman}, the shuffle product in any of these new bases can be expressed
easily in the same basis, as the shuffle product of words. On the
negative side, the description of the motivic coproduct in these bases
is not clear, and the reduction of standard multiple zeta values as a
linear combination is not simple either.

The two special cases with parameters $u,v$ being $1,0$ or $0,1$ are
specially interesting, as the conjectural bases obtained are then made
of some arborified multiple zeta values, as studied in \cite{manchon,
  clavier, ono}.

The construction of the sub-algebras $C_{u,v}$ is rather simple, as
the zinbiel sub-algebras generated inside $\As$ by two chosen elements
$z_2$ and $z_3$ in degrees $2$ and $3$. To show that they have the
correct dimensions, one just needs to prove that they are free as
zinbiel algebras. Similar results about freeness of zinbiel
sub-algebras of free zinbiel algebras have been obtained in \cite{free_dual_leibniz}.

It is possible that the same kind of ideas could be applied to some
variants of multiple zeta values, in particular to the alternating
multiple zeta values.

Acknowledgments: thanks to Francis Brown and Clément Dupont for their interest and useful suggestions.




\section{Zinbiel algebras}

Some references on zinbiel algebras are \cite{loday_dual_leibniz, dokas, loday_dialgebras}.

A zinbiel algebra over a commutative ring $R$ is a module $L$ over $R$
endowed with a bilinear product $\prec\, : L\otimes_R L \to L$ such that
\begin{equation}
  \label{axiom}
  (x \prec y) \prec z = x \prec (y \prec z) + x \prec (z \prec y)
\end{equation}
for all $x,y,z$ in $L$. It is then convenient to introduce the symmetrized product $\sha$ defined by
\begin{equation}
  x \sha y = x \prec y + y \prec x,
\end{equation}
for $x,y$ in $L$. It can be deduced from \eqref{axiom} that $\sha$ is
always a commutative and associative product on $L$.

We will always denote by $\prec$ the zinbiel product in a zinbiel algebra.

The free zinbiel algebra on a finite set $S$ over a field
$k$ has a very neat and simple description. The underlying vector
space has a basis indexed by non-empty words with letters in $S$. The
zinbiel product of two words $w$ and $w'$ is the sum of words in
the standard shuffle product of $w$ and $w'$ in which the first letter
comes from the first letter of $w$. Here is one way to remember this
rule: the symbol $\prec$ is pointing towards the word whose first letter
remains the first letter.

Let us introduce a convenient notation. For $a_1,\dots,a_n$ elements
of a zinbiel algebra, let
\begin{equation}
  \sh(a_1,\dots,a_n) = a_1 \prec \sh(a_2, \dots, a_n)
\end{equation}
be their right-parenthesized product, with $\sh(a_n) = a_n$ by convention.

In the free zinbiel algebra over a finite set $S$, the basis element
indexed by a word $(s_1,\dots,s_n)$ with letters in $S$ is exactly $\sh(s_1,\dots,s_n)$.

\section{Algebra of convergent words in $0,1$}

Let $A$ be the free zinbiel algebra on two generators $0$ and $1$.

By the general description of free zinbiel algebra recalled above, $A$
has a basis indexed by non-empty words in $0,1$. In this basis, the
product $\prec$ is one-half of the shuffle product. For example,
\begin{equation*}
 (\mathbf{1}0) \prec (10) = (\mathbf{1}010) + 2 (\mathbf{1}100),
\end{equation*}
where we have emphasized the letter that remains the first letter.
The associated commutative product $\sha$ is the standard shuffle
product. The unit for the commutative product must be added as the empty word.

The algebra $A$ is bigraded, by the number of $0$ and the number of
$1$ in a word.

Let $\mathscr{A}_{1,0}$ be the set of words starting with $1$ and
ending with $0$. Let $\As$ be the sub-space of $A$ spanned by
these words. Then clearly $\As$ is a zinbiel sub-algebra of $A$.

The algebra $\As$ is not a free zinbiel algebras, because of the
following relations. For every pair of words $x$ and $y$ in
$\mathscr{A}_{1,0}$, there holds
\begin{equation}
  \label{rels}
  (1x) \prec y = (1y) \prec x.
\end{equation}
Indeed, when seen in $A$, this equality becomes
\begin{equation*}
  (1 \prec x) \prec y = (1 \prec y) \prec x,
\end{equation*}
which is a consequence of the zinbiel axiom \eqref{axiom}.

\section{Free sub-algebras}

Our aim is to build, inside the algebra $\As$, free zinbiel
sub-algebras on two generators.

More precisely, let $u$ and $v$ be two parameters, not both zero and define
\begin{equation}
  z_2 = (1,0)\quad\text{and}\quad z_3 = u(1,0,0)+v(1,1,0)
\end{equation}
in $\As$.

Let $C_{u,v}$ be the zinbiel sub-algebra of $\As$ generated by
$z_2$ and $z_3$. Note that it only depend on the class of $(u,v)$ in
the projective line $\PP^1$.

As both $z_2$ and $z_3$ are homogeneous with respect to the total
grading of $\As$ by the length of words, $C_{u,v}$ inherits a grading where $z_2$ has
degree $2$ and $z_3$ has degree $3$.


For a word $w=(w_1,\dots,w_n)$ in the alphabet $\{2,3\}$, let us denote
\begin{equation}
  \sh_{u,v}(w) = \sh(z_{w_1},\dots, z_{w_n}).
\end{equation}

\begin{thm}
  For all $(u,v)$ not both zero, the sub-algebra  $C_{u,v}$ is a free zinbiel algebra on two generators over $\QQ$.
\end{thm}

\begin{proof}
  Using the grading of $C_{u,v}$, it is enough to prove that the
  elements $K_{u,v}(w)$ for all words $w$ in $2$ and $3$ of any given sum
  are linearly independent in $\As$. By using the bigrading of
  $\As$, one can instead prove the linear independence of the
  leading terms of these elements $K_{u,v}(w)$ with respect to either
  grading.

  If $u \not= 0$, the leading term of $K_{u,v}(w)$ with respect to the
  number of $0$ is a non-zero multiple of the element $K_{1,0}(w)$.

  If $v \not= 0$, the leading term of $K_{u,v}(w)$ with respect to the
  number of $1$ is a non-zero multiple of the element $K_{0,1}(w)$.

  It is therefore enough to prove the statement in the cases
  $(u,v)=(1,0)$ and $(u,v)=(0,1)$. This is done in the next two sections.
\end{proof}

Note that these two cases are really distinct, as there is no zinbiel
automorphism of $\As$ that would exchange them.

\begin{lemma}
  \label{shuffle_A10}
  Let $w_1$, $w_2$, \dots, $w_n$ be words in $\As$. Then
  $\sh(w_1,\dots,w_n)$ is the sum over all shuffles of the words $w_i$
  such that the first letters remain in the same order.
\end{lemma}
This is easily proved by induction, starting from the definition of $\prec$.

These shuffles will be called \textbf{good-shuffles}. The identity
shuffle is always a good-shuffle.

\subsection{Words $10$ and $100$}

Let $C_{10,100}$ be the zinbiel sub-algebra of $\As$ generated by
the words $10$ and $100$. Let us denote in this section the word $10$ by $2$ and the
word $100$ by $3$. The algebra $C_{10,100}$ is bigraded by the number
of $2$ and the number of $3$, as $10$ and $100$ are homogeneous in
$\As$ with linearly independent bidegrees.

For a word $w$ in the alphabet $\{2,3\}$, let $\cat w$ be the word in
$\As$ obtained from $w$ by the substitution $2 \mapsto 10$ and
$3 \mapsto 100$.

By Lemma \ref{shuffle_A10}, for a word $w = w_1\dots w_n$ in
$\{2,3\}$, the expansion of $\sh(w)=\sh(w_1,\dots,w_n)$ is a sum of
words in $\{0,1\}$ with the following property: the letters $1$ (one
in each $w_i$) remain in the same order. Moreover, every letter $0$
coming from $w_i$ is placed somewhere on the right of the letter $1$
coming from $w_i$.

\begin{thm}
  The zinbiel algebra $C_{10,100}$ is free over $\ZZ$.
\end{thm}

\begin{proof}
  Because $C_{10,100}$ is generated by $2$ and $3$, its dimensions are
  bounded by the dimensions of the free zinbiel algebra in two
  generators.

  It is therefore enough to prove that the dimension of the homogeneous
  component of any given bidegree $(k,\ell)$ with respect to $2$ and
  $3$ is at least the number of words with $k$ letters $2$ and $\ell$
  letters $3$.

  Let us fix $(k,\ell)$ and consider the set $S$ of words with $k$
  letters $2$ and $\ell$ letters $3$. Let us endow $S$ with the
  lexicographic order induced by the ordering of letters $2 < 3$.

  Let $M$ be the square matrix with rows and columns indexed by $S$
  with coefficient $M_{w,w'}$ being the number of occurrences of the word
  $\cat w'$ in the expansion of $\sh(w)$ as a sum of words.

  Let us prove that $M$ is upper triangular with $1$ on the diagonal.

  \begin{lemma}
    \label{lemme_facile}
    Let $w,w' \in S$ such that $M_{w,w'} \not= 0$ and $w$ shares a
    prefix with $w'$. Then the restriction to the common prefix of any
    good-shuffle that maps $\cat w$ to $\cat w'$ is the identity.
  \end{lemma}
  \begin{proof}
    This is proved by induction on the length $i$ of the common
    prefix, starting from the empty prefix. Let us assume the induction hypothesis before $w_i$ and
    that moreover $w_i = w'_i$. The $1$ in $w'_i$ must come from the
    $1$ in $w_i$ because the $1$'s remain in the same order. Then the
    $0$ (one or two) from $w'_i$ must be to the right of their
    associated $1$, which must therefore be the $1$ coming from
    $w_i$. The only possible way is that these $0$ are not shuffled
    and remain at their initial positions in $w_i$.
  \end{proof}

  In the case $w=w'$, this lemma implies that the diagonal of $M$ is
  made of $1$'s.

  Now consider $w \not= w' \in S$ such that $M_{w,w'} \not= 0$ and the
  first different letter happens at position $i$, where $w_i \not= w'_i$.
  Assume by contradiction that $w_i = 2$ and $w'_i=3$. By the lemma \ref{lemme_facile},
  any good-shuffle that maps $\cat w$ to $\cat w'$ is the identity on the
  common prefix of $w$ and $w'$. Therefore the letter $1$ in $w'_i$
  comes from the letter $1$ in $w_i$. But then the two $0$ in $w'_i$
  need to have their associated $1$ on their left, which is not
  possible.
\end{proof}



\subsection{Words $10$ and $110$}

Let us now turn to the other case, slightly more complicated.

Let $C_{10,110}$ be the zinbiel sub-algebra of $\As$ generated by
the words $10$ and $110$. Let us denote in this section the word $10$ by $2$ and the
word $110$ by $3$. The algebra $C_{10,110}$ is bigraded by the number
of $2$ and the number of $3$, as $10$ and $110$ are homogeneous in
$\As$ with linearly independent bidegrees.

For a word $w$ in the alphabet $\{2,3\}$, let $\cat w$ be the word in
$\As$ obtained from $w$ by the substitution $2 \mapsto 10$ and
$3 \mapsto 110$.

\begin{thm}
  The zinbiel algebra $C_{10,110}$ is free over $\QQ$.
\end{thm}
\begin{proof}
  Because $C_{10,110}$ is generated by $2$ and $3$, its dimensions are
  bounded by the dimensions of the free zinbiel algebra in two
  generators.

  It is therefore enough to prove that the dimension of the homogeneous
  component of any given bidegree $(k,\ell)$ with respect to $2$ and
  $3$ is at least the number of words with $k$ letters $2$ and $\ell$
  letters $3$.

  Let us fix $(k,\ell)$ and consider the set $S$ of words with $k$
  letters $2$ and $\ell$ letters $3$. Let us endow $S$ with the
  lexicographic order induced by the ordering of letters $2 < 3$.

  Let $M$ be the square matrix with rows and columns indexed by $S$
  with coefficient $M_{w,w'}$ being the number of occurrences of the word
  $\cat w'$ in the expansion of $\sh(w)$ as a sum of words.

  Let us prove that $M$ is lower triangular with no zero on the diagonal.

  Let us first remark that the diagonal coefficient for a word $w$ in
  $S$ counts the number of good-shuffles that preserves $\cat w$. But
  this set always contains the identity shuffle.

  \begin{lemma}
    \label{lemme_complexe}
    Let $w,w' \in S$ such that $M_{w,w'} \not= 0$ and $w$ shares a
    prefix $w_1\dots w_i$ with $w'$. Let $\sigma$ be any good-shuffle
    that maps $\cat w$ to $\cat w'$. Then
    \begin{itemize}
    \item[(i)] either $\sigma$ stabilizes the common prefix,
    \item[(ii)] or the following statements hold:
      
      - There exists a letter $3$ in the prefix. Let $w_k$ be the
      rightmost such letter. The shuffle
      $\sigma$ stabilizes the prefix before $w_k$.

      - The letter $w_{i+1}$ is $2$.

      - Between $w_k$ and $w_{i+1}$, the shuffle $\sigma$ acts like this:
      \begin{equation*}
        \begin{array}{rr|rrrr|r}
             & \dots & w_k = 3 & 2 & \dots & w_i = 2 & w_{i+1} = 2\\
          w  & \dots & {\color{red}1}{\color{blue}10} & 10 & \dots & 10 & 10\\
          w' & \dots & {\color{red}1}10 & 10 & \dots & 10 & \\
        \end{array}
      \end{equation*}
      where each $10$ (from a letter $2$) displayed in the line $w$ is
      send to the final $10$ (from a letter $3$ or $2$) in the
      previous term of the line $w'$. The first ${\color{red} 1}$ of
      $w_k$ is sent to the first ${\color{red} 1}$ of $w'_k$ and the
      final ${\color{blue}10}$ of $w_k$ is sent in the suffix of $w'$
      after $w'_i$.
  \end{itemize}
  \end{lemma}
  \begin{proof}
    This is proved by induction on the length of the common prefix,
    starting with the empty prefix.

    Assume first the induction hypothesis with condition $(i)$ before
    $w_i$ and moreover $w_i = w'_i$. Necessarily, the first $1$ in
    $w'_i$ must come from the first $1$ in $w_i$, as the order is
    preserved on the first letters.

    If $w_i$ is $2$, the $0$ in $w'_i$ must come from the $0$ in
    $w_i$, as the only available $1$ to its left is that of $w'_i$. So
    condition $(i)$ holds for the extended prefix.

    If $w_i$ is $3$, and if the second $1$ in $w'_i$ comes from $w_i$
    too, one finds that condition $(i)$ holds for the extended prefix,
    for the same reason as in the previous case.

    Otherwise, the second $1$ in $w'_i$ must come from $w_{i+1}$. And
    the $0$ in $w'_i$ must be preceded by all the associated $1$'s. This
    implies that this $0$ comes from $w_{i+1}$ and that $w_{i+1} = 2$.
    All this gives condition $(ii)$ for the extended prefix, in
    the special situation where $k = i$.

    Assume now the induction hypothesis with condition $(ii)$ before
    $w_i$ and moreover $w_i = w'_i$. Then $w_{i}$ is a $2$.

    Assume first that $\sigma$ sends the second $1$ of $w_k$ to the $1$
    in $w'_i$.  Then the $0$ in $w_k$ must be sent to the $0$ in
    $w'_i$.  Therefore in this case, one obtains condition $(i)$ for
    the extended prefix.
    
    Otherwise, $\sigma$ sends the second $1$ of $w_k$ somewhere in the
    suffix of $w'$ after $w'_i$. Then the first $1$ of $w_{i+1}$ must
    be sent to the first one of $w'_i$.  And
    the $0$ in $w'_{i}$ must be preceded by all the associated $1$'s. This
    implies that this $0$ comes from $w_{i+1}$ and that $w_{i+1} = 2$.

    The first statement of condition $(ii)$ holds by induction. We
    just proved the two other statements, so condition $(ii)$ holds
    for the extended prefix.
  \end{proof}


  Now consider $w \not= w' \in S$ such that $M_{w,w'} \not= 0$ and the
  first different letter happens at position $i$ where $w_i \not= w'_i$.
  Assume by contradiction that $w_i = 3$ and $w'_i=2$. Let $\sigma$ be
  any good-shuffle that maps $\cat w$ to $\cat w'$.

  By the lemma \ref{lemme_complexe}, either condition $(i)$ or
  condition $(ii)$ holds for $\sigma$.

  Condition $(ii)$ cannot hold because $w_i = 3$.
  
  Therefore condition $(i)$ holds. The unique letter $1$ in $w'_i$ comes
  from the first letter $1$ in $w_i$. But then the $0$ in $w'_i$
  either comes from $w_i = 3$ or from some $w_j$ with $j > i$. In both
  cases, this $0$ has not enough $1$'s on its left.
\end{proof}


Remark: One can deduce from the proof of lemma \ref{lemme_complexe} a more precise description of
the diagonal coefficients of the matrix $M$. The coefficient of a word $w$
in $2$ and $3$ is the product of the lengths of the blocks in the
unique factorization of $w$ into blocks $322\dots2$, omitting the
possible initial sequence of $2$. For example, for $(2, 3, 3, 2, 3, 2)$, one gets $4 = 1 \times 2 \times 2$.

\subsection{Remarks and questions}

One could ask the same question of freeness about several larger zinbiel sub-algebras of $\As$.

Sometimes the answer is clearly  negative, for the same reasons as for
$\As$. This is for instance the case of the sub-algebra generated
by the words $10,110,1110$ which contains one relation
\eqref{rels}.

The following cases may be free, as these algebras do not contain any obvious relation of this kind.
\begin{itemize}
  \item[(A)] the sub-algebra generated by words $10, 100, 110$,
  \item[(B)] the sub-algebra generated by words of the shape $10^\ell$,
  \item[(C)] the sub-algebra generated by $110$ and words of the shape $10^\ell$,
  \item[(D)] the sub-algebra generated by words not starting by $11$,
  \item[(E)] the sub-algebra generated by $110$ and words not starting by $11$.
\end{itemize}

Some closely related questions have been answered in \cite[\S
4]{free_dual_leibniz}.

\section{Quotient map to motivic multiple zeta values}

We will use the following convention for formal iterated integrals:
\begin{equation}
  \label{iterated}
  I(\ep_1, \dots, \ep_k) = \mathop{\int\cdots\int}\limits_{0<t_1<\cdots <t_k<1} \omega_{\ep_1}(t_1) \cdots \omega_{\ep_k}(t_k),
\end{equation}
where each $\ep_i$ is either $0$ or $1$, with $\ep_1=1$ and $\ep_k=0$.

We will denote motivic multiple zeta values $\zeta(n_1,\dots,n_k)$, with the convention
\begin{equation}
  \label{conversion}
  \zeta(k_1,\dots,k_r) = I(1,0^{k_1-1},1,0^{k_2-1},\dots,1,0^{k_r-1})
\end{equation}
for $r \geq 1$. Here a power of $0$ means a repeated $0$.

Let us denote by $\Pi$ the surjective quotient map from $\As$ to
$\AMZV$, whose kernel is the ideal of motivic relations between formal
iterated integrals.

For $(u,v)$ not both zero, the space $C_{u,v}$ is a zinbiel
sub-algebra of $\As$, hence also a commutative sub-algebra
of $\As$ for the symmetrized product, which is just the shuffle
product.

For every choice of parameters $(u,v)$ not both zero, one can
therefore restrict $\Pi$ to this commutative sub-algebra $C_{u,v}$ of
$\As$. This gives a morphism of commutative graded algebras $\Pi$ from
$C_{u,v}$ to $\AMZV$. Note that these two graded algebras have the
same generating series $F = 1/(1-x^2-x^3)$.

One can therefore wonder if the morphism $\Pi$ could be an
isomorphism, under some conditions on $(u,v)$. As the algebra
$C_{u,v}$ itself, this property only depend on the projective class of
$(u,v)$ in $\PP^1$.

Let us consider first the very special case where $u+v=0$. In this
case, the image of the word $z_3$ in $C_{u,v}$ is given by
$\Pi(z_3) = u \zeta(3) + v \zeta(1,2)$. But it is known that
$\zeta(3) = \zeta(1,2)$ in $\AMZV$, hence $\Pi(z_3) = 0$. Therefore
$\Pi$ is not surjective in this case.

Using a computer, one can compute the first few graded dimensions of the image.
\begin{conjecture}
  When $u+v = 0$, the image of $C_{u,v}$ by $\Pi$ is a sub-algebra of
  $\AMZV$ with generating series $1 + x^2 F$.
\end{conjecture}
This is the generating series of the quotient algebra of $\AMZV$ by $\zeta(3)$,
because
\begin{equation}
  1 + x^2 F = (1 - x^3) F.
\end{equation}
One can wonder what could be, for this sub-algebra, an analog of the
famous conjecture of Broadhurst-Kreimer describing the dimensions of
$\AMZV$ according to both weight and depth.

Note also the similarity with the quotient algebra of $\AMZV$ by
$\zeta(2)$, which has generating series
\begin{equation}
  1 + x^3 F = (1 - x^2) F
\end{equation}
and appears in the motivic coproduct and in the
study of $p$-adic multiple zeta values \cite{furusho_padic1, furusho_padic2}.

\medskip

Excluding from now on the special case $u+v=0$, one can assume without
loss of generality that $u + v = 1$ and set $v = 1 - u$. Then
$\Pi(z_3)=\zeta(3)$.

When using $u$ as a formal parameter, one expects the following statement.
\begin{conjecture}
  The morphism $\Pi$ from $C_{u,1-u}$ to $\AMZV$ is generically an
  isomorphism of commutative algebras.
\end{conjecture}

Let us say that a value of $u$ is \textbf{non-singular} if this
statement holds at $u$, and \textbf{singular} otherwise. For every non-singular $u$, the isomorphism
$\Pi$ defines a bigrading of the algebra $\AMZV$. Moreover, one gets a
basis in $\AMZV$ from the basis of $C_{u,1-u}$ made of words in $z_2$
and $z_3$.

More and more singular values appear when considering the restriction
of $\Pi$ at increasing weights. One could still hope that some specific
values of $u$ are non-singular, for example $u=0$ and $u=1$. So far,
no non-singular value of $u$ is known.

Let us describe the first few polynomials whose zeroes are singular
values, in their order of apparition, where $n$ is the weight.
\begin{equation*}
  \begin{array}{rr}
    n & p\\
    5 & 5 u - 6 \\
    7 & 14 u + 51 \\
    8 & 27 u^2 - 26 u + 10 \\
    9 & 865 u - 4164 \\
    10 & 2011 u^2 - 3381 u + 1581 \\
    11 & 461516 u^4 - 3721029 u^3 + 7046644 u^2 - 6169912 u + 2357966 \\
    12 & 207786 u^4 - 185687 u^3 - 1076020 u^2 + 1483088 u - 562680 \\
  \end{array}
\end{equation*}
One can check that $0$ and $1$ are not roots of any of these polynomials.


\begin{figure}[h]
  \centering
  \includegraphics[width=12cm]{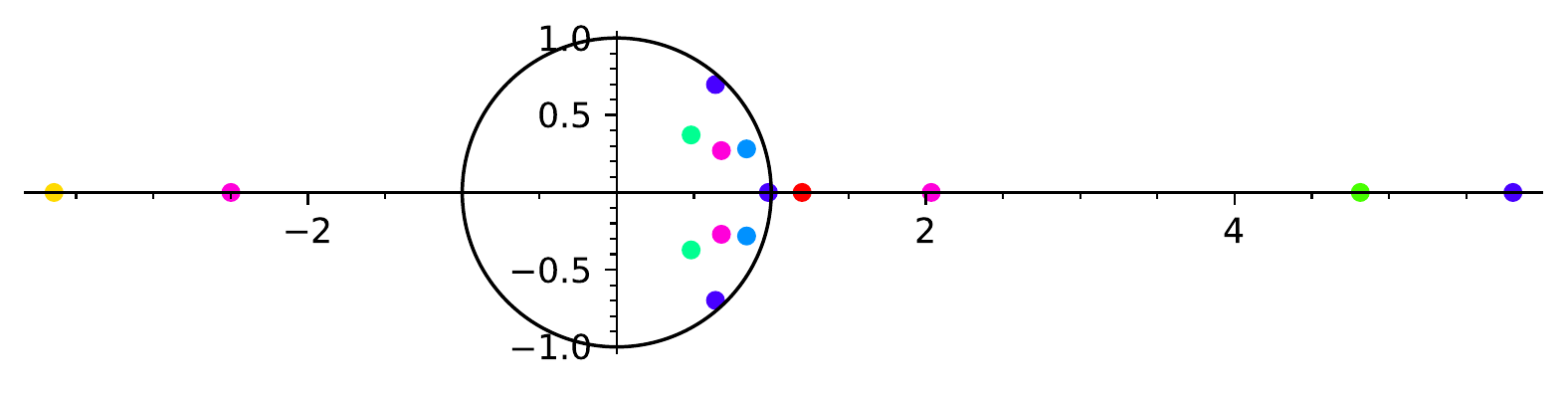}
  \caption{Singular values of $u$ in weight at most $12$.}
  \label{fig:singular}
\end{figure}









Let us give the first few images by the morphism $\Pi$ from $C_{u,1-u}$
to $\AMZV$ of short words in $z_2$ and $z_3$.

\begin{align}
  B(2) &= \zeta(2)\\
  B(3) &= \zeta(3)\\
  B(2,2) &= 2\zeta(1,3) + \zeta(2,2)\\
  B(2,3) &= \left(u + 2\right)\zeta(1,4) + \zeta(2,3) + \left(-u + 1\right)\zeta(3,2)\\
  B(3,2) &= \left(-u + 4\right)\zeta(1,4) + 2\zeta(2,3) + u\zeta(3,2).
\end{align}

One can then check that $ B(2) B(3) = B(2,3) + B(3,2)$, as a simple
example of the general rule that the product in the conjectural bases
is given by the shuffle product.

\section{Variants of multiple zeta values}

There are some other situations where one could try to apply the same ideas.

A first example is given by the algebra of alternating multiple zeta
values, defined as the iterated integrals of the following three
$1$-forms:
\begin{equation}
  \omega_0 = \frac{dt}{t}, \omega_{-1}=\frac{dt}{t-1}, \omega_1=\frac{dt}{t+1}.
\end{equation}
A conjecture due to Broadhurst (see \cite{broadhurst_kreimer, broadhurst}) states
that the graded dimensions of this algebra are given by the Fibonacci
numbers, with generating series $1/(1-x-x^2)$.

One could therefore consider the zinbiel sub-algebra of the free zinbiel
algebra on $\{-1,0,1\}$ generated by the abstract iterated integrals
$I(-1)$ and $I(1,0)$. Is this a free zinbiel algebra on these generators ?

\medskip

A similar case is the algebra of multiple Landen values, defined in
\cite{broadhurst_landen} as iterated integrals of the following
$1$-forms: $A=dx/x$, $B=dx/(1-x)$, $F=dx/(1-\rho^2x)$ and
$G=dx/(1-\rho)$ where $\rho$ is the golden ratio.

Broadhurst conjectured in \cite{broadhurst_landen} that the generating
series for this algebra is $1/(1-x-x^2-x^3)$, whose coefficients are
tribonacci numbers. The exact same formula is also expected to give
the graded dimensions of the sub-algebra of iterated integrals of $A$
and $G$.

The first few terms are
\begin{equation*}
  1, 1, 2, 4, 7, 13, 24, 44, 81, 149, 274, 504, 927, 1705,\dots
\end{equation*}
and the expected bases, as words in $A$ and $G$ are $\{G\}$ in degree $1$;
$\{AG,GG\}$ in degree $2$ and $\{AGG,AAG,GAG,GGG\}$ in degree $3$.

In this setting, one could look for a free zinbiel sub-algebra on
three generators of degrees $1,2$ and $3$ inside the free zinbiel algebra on
$A$ and $G$.

Another interesting case to consider would be multiple Watson values \cite{broadhurst_watson}.

\bibliographystyle{plain}
\bibliography{article_free_shuffle.bib}

\end{document}